\documentclass[a4,12pt]{article}

 \usepackage{graphicx}
\usepackage{color}

\usepackage{amsmath, amssymb, amsthm, enumerate,commath}

\usepackage{babel}
\usepackage{amsmath,amssymb,amsthm,enumerate,mathrsfs}
\newtheorem{theorem}{Theorem}[section]
\newtheorem{corollary}[theorem]{Corollary}
\newtheorem{lemma}[theorem]{Lemma}

\theoremstyle{definition}
\newtheorem{definition}[theorem]{Definition}
\newtheorem{example}[theorem]{Example}

\newtheorem{remark}[theorem]{Remark}

\numberwithin{equation}{section}

\begin{document}
\makeatletter

\begin{center}
\large {\bf Best Approximate Solution for Generalized Nash Games and Quasi-optimization Problems}
\end{center}\vspace{5mm}

\begin{center}
	\renewcommand{\thefootnote}{\fnsymbol{footnote}}
	\textsc{Asrifa Sultana\footnote[1]{Corresponding author.  e-mail: asrifa@iitbhilai.ac.in}\footnote[2]{Department of Mathematics, Indian Institute of Technology Bhilai, Raipur - 492015, India.}, Shivani Valecha\footnotemark[2]}
\end{center}
\vspace{2mm}
\footnotesize{
\noindent\begin{minipage}{14cm}
{\bf Abstract:}
In this article, we consider generalized Nash games where the associated constraint map is not necessarily self. The classical Nash equilibrium may not exist for such games and therefore we introduce the notion of best approximate solution for such games. We investigate the occurrence of best approximate solutions for such generalized Nash games consisting of infinitely many players in which each player regulates the strategy variable lying in a topological vector space. Based on the maximum theorem and a fixed point result for Kakutani factorizable maps, we derive the existence of best approximate solutions under the quasi-concavity and weak continuity assumption on players' objective functions. Furthermore, we demonstrate the occurrence of best approximate solutions for quasi-optimization problems. 
\end{minipage}
 \\[5mm]


\hbox to14cm{\hrulefill}\par


\section{Introduction}\label{section1}
The well known Nash equilibrium problem (NEP) introduced by John F. Nash \cite{nash1}, was defined for finite number of players, each having a fixed set of strategies. The \textit{generalized Nash equilibrium problem} (GNEP) is a generalization of NEP in which the strategy set of any player is allowed to rely on the strategies selected by the rival players. The theory of GNEP was established by Debreu \cite{debreu1952} and further it was studied by Arrow and Debreu \cite{debreu} in connection with economic models. Since then the study of GNEP is gaining popularity due to its wide applicability in mathematical economy \cite{debreu}, electricity market \cite{ausselTDGNEP,aussel2016}, models related to environmental sustainability \cite{ramos} and many more. For the brief history of GNEP and the concise overview of recent advances in this field, the reader may refer to a survey paper by Facchinei and Kanzow \cite{faccheini} and work done by Fischer et al. \cite{fischer}. It can be observed that Debreu \cite{debreu1952} considered generalized Nash games for finitely many players with Euclidean spaces as strategy spaces. Later on, the generalized Nash games for finitely many players having infinite dimensional strategy spaces were studied by M. Lassonde \cite{lassonde}, Aussel et al. \cite{ausselTDGNEP} and others. 
However, Yannelis and Prabhakar \cite{yannels} and Tian and Zhou \cite{tian} considered the GNEP for infinitely many players having topological vector spaces as strategy space.

In a non-cooperative game, let us assume $P_i$ denotes the $i^{th}$ player in an (possibly) uncountable index set $I$. \color{black}Suppose the player $P_i$ regulates the strategy variable $x_i$ which lies in a subset $X_i$ of a locally convex \textit{topological vector space} (t.v.s.) $Y_i$ and $x_{-i}$ denotes the vector consisting of the strategies selected by rival players. Then $x=(x_{-i},x_i)\in X=\prod_{i\in I}X_i$ represents the whole vector of strategies. Let $u_i:X=\prod_{i\in I}X_i \rightarrow \mathbb{R}$ be the real valued objective/payoff function for the player $P_i$. Then for the given $x_{-i}$ (the strategies of rival players), the strategy set of a player $P_i$ is confined to a feasible strategy set $F_i(x_{-i})\subseteq X_i$ and the objective of the $i^{th}$ player is to select some suitable strategy $x_i\in F_i(x_{-i})$ such that $x_i$ solves the below stated problem,
$$ P_i(x_{-i}):\qquad\max_{z_i} \, u_i(x_{-i},z_i)\text{ subject to } z_i\in F_i(x_{-i}).$$
Suppose that $S_i(x_{-i})$ is the solution set of the above problem $P_i(x_{-i})$ for the given strategies $x_{-i}$ selected by the rival players. Then $\overline x=(\overline x_i)_{i\in I}\in X$ is an equilibrium point for the considered GNEP if $\overline x_i\in S_i(\overline x_{-i})$ for any $i\in I$. Moreover, the described GNEP reduces to NEP if for each $x_{-i}$, the set $F_i(x_{-i})$ equals to some fixed $C_i\subseteq X_i$ .


In \cite{tian}, Tian-Zhou studied the generalized Nash games for infinitely many players having topological vector spaces as strategy spaces. 
Based on an extended maximum theorem \cite[Theorem 1]{tian}, they established the below stated existence result for GNEP, by relaxing the lower semi-continuity of the defined constraint map and payoff functions:
\begin{theorem}
	\cite[Theorem 2]{tian}\label{theorem6}
	Suppose $X_i\neq \emptyset$ is a compact and convex subset of a locally convex Hausdorff t.v.s. for any $i$ in (possibly) uncountable index set $I$. If for each $i\in I$,
	\begin{itemize}
		\item[a)] $F_i:X_{-i}\rightarrow 2^{X_{i}}$ is non-empty, convex and compact valued upper semi-continuous map;
		\item[b)] $u_i:X_{-i}\times X_i \rightarrow \mathbb{R}$ is upper semi-continuous function which is quasi-concave in $x_{i}$;
		\item[c)]  $u_i$ is feasible path transfer lower semi-continuous function in $x_{-i}$ refer to $F_{i}$.
	\end{itemize} 
	Then the given generalized Nash game admits atleast one equilibrium point.
\end{theorem}

On the other hand, Aussel et al. \cite{aussel2016} recently considered $n$-person generalized Nash game having non-self constraint map, that is, for given strategy $x_{-i}$ of rival players $F_i(x_{-i})$ is not necessarily subset of $X_i$ and the product map  $F=\prod_{i=1}^{n}F_i$ need not be self. In fact, they considered $F_i(x_{-i})\subseteq Y_i$ and $F_i(x_{-i})\cap X_i$ is possibly empty for any $i\in \{1,2,\cdots n\}$. It is worth mentioning that the classical Nash equilibrium may not exist for such games. They have provided an illustration of deregulated electricity market model which motivated them to introduce the notion of projected solution for such GNEP. The authors demonstrated the existence of projected solution for the $n$-person GNEP having Euclidean spaces as strategy spaces. Furthermore, Bueno and Cotrina \cite{cotrinaGNEP} established the existence of projected solutions for $n$-person generalized Nash games having non-self constraint maps defined over Banach spaces. 

Following Aussel et al. \cite{aussel2016} and Bueno-Cotrina \cite{cotrinaGNEP}, we consider generalized Nash games consisting of infinitely many players where the associated constraint map is not necessarily self and initiate a concept of best approximate solution for such games. Our primary objective is to investigate the existence of best approximate solutions for the generalized Nash games consisting of infinitely many players in which each player regulates a strategy variable lying in a locally convex t.v.s. 
We establish the existence of best approximate solutions for such games under the quasi-concavity and weak continuity assumption on players' objective functions. 
In particular, our result extends the above stated Theorem \ref{theorem6} to the case where the constraint map need not be self. Finally, we show the occurrence of best approximate solutions for quasi-optimization problems. 

This article is arranged in the following sequence. In section \ref{section2.1}, we introduce the concept of best approximate solution for GNEP. Section \ref{section3} consists of basic definitions, notations and preliminary results. Based on the generalized Berge's maximum theorem \cite{tian} and a fixed point result for Kakutani factorizable maps \cite{lassonde}, we prove the existence of best approximate solutions for generalized Nash equilibrium problems defined over infinite dimensional strategy spaces in section \ref{section4.1}. Furthermore, in section \ref{section4.2} we demonstrate the occurrence of best approximate solutions for quasi-optimization problems.

\section{Concept of Best Approximate Solution}\label{section2.1}
Now we introduce the concept of best approximate solution for generalized Nash equilibrium problem with non-self constraint map defined over infinite dimensional strategy spaces. For any $i$ in an uncountable index set $I$, suppose $X_i$ is a subset of a locally convex t.v.s. $Y_i$. Consider, 
$$ Y=\prod_{i\in I}Y_i,~X=\prod_{i\in I}X_i,~ Y_{-i}=\prod_{j\in I\setminus \{i\}} Y_j,~X_{-i}=\prod_{j\in I\setminus \{i\}} X_j,$$
where $X_{-i}$ denotes the cartesian product of sets other than $X_i$. In the classical definition of GNEP, it is assumed that for a given strategy vector $x_{-i}$ of rival players, $i^{th}$ player selects a strategy variable $x_i$ in $F_i(x_{-i})\subset X_i$. Let us consider a more general situation that the feasible strategy set $F_i(x_{-i})\subset Y_i$ and $F_i(x_{-i})\cap X_i$ is possibly empty for any $i\in I$ (see \cite{aussel2016,cotrinaGNEP}). Then, in this case the classical Nash equilibrium may or may not exists (refer to Example \ref{ex4}) and therefore we initiate the notion of best approximate solution for such GNEP.
\begin{definition}
	Let $I$ be an (possibly) uncountable index set of players. 
	For any $i\in I$, let $X_i$ be a subset of a locally convex t.v.s. $Y_i$ and $p_i$ be a continuous semi-norm defined over $Y_i$. 
	Suppose $u_i:Y \rightarrow \mathbb{R}$ and $F_i:X_{-i}\rightarrow 2^{Y_i}$ denotes objective function and feasible strategy map for an $i^{th}$ player. Then the strategy vector $\tilde x=(\tilde x_i)_{i\in I} \in X$ is called \textit{best approximate solution} for $GNEP(F_i,u_i)_{i\in I}$ iff there exists $\tilde y=(\tilde y_i)_{i\in I} \in Y$ satisfying,
	\begin{itemize}
		\item[(a)] $p_i(\tilde y_i-\tilde x_i)= \displaystyle \inf_{x_i\in X_i} p_i(\tilde y_i-x_i)$ for each $i\in I$.
		\item[(b)] $(\tilde y_i)_{i\in I}$ solves the following maximization problem for any $i\in I$, 
		$$ u_i(\tilde x_{-i}, \tilde y_i) \geq u_i(\tilde x_{-i}, z_i),~\text{for all}~ z_i\in F_i(\tilde x_{-i}).$$
	\end{itemize}
\end{definition}\color{black}
The reader may check that every generalized Nash equilibrium $\tilde x\in X$ is best approximate solution for the considered GNEP, as we may assume $\tilde y=\tilde x$ in above definition. 
On the other hand, every best approximate becomes a generalized Nash equilibrium for the considered game 
if $F=\prod_{i\in I} F_i$ is a self map. For better understanding the proposed concept of best approximate solution, we now present the generalized Nash game between two players.
\begin{example}\label{ex4}
	In a non-cooperative game, suppose players $P_1$ and $P_2$ control the strategy variable $\textbf x=(x_1,x_2),\textbf y=(y_1,y_2)$ in $\mathbb{R}^2$. Let $X$ and $Y$ be non-empty subsets of $\mathbb{R}^2$ defined as follows,
	$$X=Y=\{(x_1,x_2)\in \mathbb{R}^2|~ (x_1,x_2)\in[0,1]\times [0,1] ~\text{and}~x_1+x_2\geq 1\}.$$
	
	Let us define $F_1:Y\rightarrow 2^{\mathbb{R}^2}$ and $F_2:X\rightarrow 2^{\mathbb{R}^2}$ as the feasible strategy maps for the players $P_1$ and $P_2$ respectively,
	\begin{align*}
	&F_1(y_1,y_2)= \bigg\{\frac{2}{\sqrt{{y_1}^2+{y_2}^2}}(y_1,y_2)+(v_1,v_2)\bigg |~ (v_1,v_2)\in[0,1]\times [0,1]\bigg\}, \\
	&F_2(x_1,x_2)= \bigg\{\frac{\sqrt 2}{\sqrt{{x_1}^2+{x_2}^2}}(x_1,x_2)+(v_1,v_2)\bigg |~ (v_1,v_2)\in[0,1]\times [0,1]\bigg\}.
	\end{align*}
	
	 For players $P_1$ and $P_2$, let us define payoff functions $u_1:\mathbb{R}^2\times \mathbb{R}^2\rightarrow \mathbb{R}$ and $u_2:\mathbb{R}^2\times \mathbb{R}^2\rightarrow \mathbb{R}$ as follows:
	\begin{align*}
	&u_1(\textbf{x},\textbf{y})=2x_1+2x_2+3y_2,\\
	&u_2(\textbf{x},\textbf{y})=2x_1+y_1+y_2.
	\end{align*}
	It is easy to observe that the above defined $GNEP(F_i,u_i)_{i\in I}$ is the generalized Nash game with non-self constraint map due to the fact $F_1(Y)\cap X=\emptyset$. Consequently, $GNEP(F_i,u_i)_{i\in I}$ does not admit any (classical) equilibrium. However, we obtain $(\textbf x^{*},\textbf y^{*})=((1,1),(1,1))$ as a best approximate solution for given $GNEP(F_i,u_i)_{i\in I}$.
\end{example}
\begin{remark}
	The above definition is inspired from the definition of projected solution for GNEP which was introduced by Aussel et al. in finite dimensional setting (see \cite{aussel2016}).
\end{remark}

\section{Preliminaries}\label{section3}
\subsection{Basic Notations and Definitions}\label{section3.1}
Assume that $f:X\rightarrow\mathbb{R}$ is a real valued function where $X$ is a subset of some topological space. For the concept of upper and lower semi-continuity of $f$, the reader may refer \cite{berge,tian}. Suppose $U$ and $V$ are topological spaces and $K:U\rightarrow2^V$ is a set-valued map. The reader is referred to \cite{berge} for the notion of upper and lower semi-continuity of the map $K$. In context to \cite{tian}, we have stated the following concept which is weaker in comparison to lower semi-continuity of the correspondence $K$ and the real valued map $f$.
\begin{definition}\cite{tian}
	Assume that $U$ and $V$ are topological spaces and $K:U\rightarrow2^V$ is a set-valued map. A function $f:U \times V \rightarrow \mathbb{R}$ is known as ``feasible path transfer lower semi-continuous" (abbreviated as FPT l.s.c.) in $u$ refer to $K$ if for any point $(u,v) \in U \times V$ with $v \in K(u)$ and for any positive $\epsilon$, an open neighbourhood $\mathcal{O}(u)$ of $u$ can be obtained such that $\forall\, u' \in \mathcal{O}(u)$, $\exists\, v'\in K(u')$ satisfying,
	$$ f(u,v) < f(u',v')+ \epsilon.$$
\end{definition}
We can observe that if the maps $K$ and $f$ are lower semi-continuous, then $f$ is feasible path transfer lower semi-continuous refer to $K$. However, the example which is stated below disproves its converse.
\begin{example}
	Let  $f:[0,1]\times\mathbb{R}\rightarrow \mathbb{R}$ be a real valued function given as,
	$$f(u,v)= \begin{cases}
	1 &v=u \\
	0 &v\neq u.
	\end{cases}
	$$
	Suppose $K:[0,1]\rightarrow2^\mathbb{R}$ is a set-valued map given as follows,
	$$K(u)=\begin{cases}
	\mathbb{R} &u=0 \\
	\{u\}  &u\neq 0,
	\end{cases}$$
	then the map $f$ is not lower semi-continuous at $(0,0)$ and $K$ is not lower semi-continuous at the point $0$ but $f$ is FPT l.s.c. in variable $u$ refer to $K$ for any $u \in [0,1]$.
\end{example}

Using this weaker lower semi-continuity condition on $K$ and $f$, the following generalization of Berge's maximum theorem is derived by Tian and Zhou \cite{tian} to ensure that the maximizing maps of the players in a given generalized Nash game meets upper semi-continuity property.

\begin{theorem}\label{theorem1}\cite[Theorem 1]{tian}
For the given topological spaces $U$ and $V$, suppose the correspondence $K:U\rightarrow 2^V$ meets the upper semi-continuity property with $K(u)$ being non-empty compact set for each $u$. Let $f:U \times V \rightarrow \mathbb{R}$ be an upper semi-continuous function satisfying the condition of feasible path transfer lower semi-continuity for each $u\in U$ refer to $K$. Then the maximizing map $M:U\rightarrow 2^V$ given as,
	\begin{center}
		$M(u)=\{ v\in K(u): f(u,v)\geq f(u,w),\forall w\in K(u)\}$
	\end{center}
	meets upper semi-continuity property with $M(u)$ being non-empty compact for any $u\in U$.
\end{theorem}

Let us now recall the below stated generalization of the celebrated Kakutani fixed point result which will be used to validate the existence of best approximate solution for generalized Nash game in the sequel.

\begin{theorem}\label{theorem2}\cite[Theorem 4]{lassonde}
	Assume that $V$ is a locally convex Hausdorff t.v.s. Suppose $U_r\, (\text{for }r=1,2,\cdots N)$ and $U$ are non-empty convex subsets of $V$. Let $\Gamma:U\rightarrow 2^U$ be a map of the form $\Gamma=\Gamma_N \Gamma_{N-1}\cdots \Gamma_0$ or $\Gamma:U=U_0\xrightarrow {\Gamma_0} U_{1}\xrightarrow{\Gamma_1} U_2 \cdots U_N\xrightarrow{\Gamma_N} U_{N+1}=U$ where $\Gamma_r:U_r\rightarrow 2^{U_{r+1}}$ meets the upper semi-continuity property with non-empty, convex and compact values for each $r=0,1,2, \cdots N$. Then there exists $u\in U$ such that $u\in \Gamma(u)$ if $\overline{\Gamma(U)}$ is compact.
\end{theorem}

We would like to recall the concept of quasi-concavity for real valued functions. Assume that $U$ is a convex subset of some t.v.s. then $f:U\rightarrow \mathbb{R}$ is called quasi-concave function \cite{tian} if for any $u,v\in U,\, 0\leq t\leq 1$ we have,
$ f(tu+(1-t)v) \geq min\{f(u),f(v)\}.$ Note that the quasi-concave functions are preferable over concave functions. Indeed if we consider $f:\mathbb{R}\rightarrow \mathbb{R}$ defined as $f(x)=x^3$ then $f$ is quasi-concave but not concave.

\subsection{Concept related to Best Approximation Theory}\label{section3.2}
Let us recall some definitions and results related to the best approximation theory. Suppose $V$ is a locally convex Hausdorff t.v.s. with continuous semi-norm $p$. Assume that $U\subseteq V$ and for any $v\in V$, let us define $d_p(v,U)=\displaystyle \inf_{u\in U} p(v-u).$ Then the set $U$ called approximatively compact (refer to $p$) if for each $v\in V$ and every net $\{u_\alpha\}\in U$ satisfying $p(v-u_\alpha)\rightarrow d_p(v,U)$, there is a subnet $\{u_\beta\}$ of $\{u_\alpha\}$ converging to some point $u\in U$. It is easy to observe that any subset of $V$ is approximatively compact if it is compact. However, the converse implication may not hold (refer to Example \ref{ex3}). For better understanding let us consider some examples of approximatively compact sets.
\begin{example}\label{ex2}\cite{singh}
	In a uniformly convex Banach space $X$, any $C\subseteq X$ is approximatively compact with respect to norm defined over $X$ if it is closed and convex.
\end{example}
\begin{example}\label{ex3}\cite{singh}
	In infinite dimensional uniformly convex Banach space $X$, the unit ball $\overline B_1=\{x\in X|\,\norm{x}\leq 1\}$ is approximatively compact with respect to norm defined over $X$ but $\overline B_1$ is not compact in $X$.
\end{example}	

The following well known result related to best approximation map will be further used by us.
\begin{lemma}\label{lemma1}\cite{vetrivel}
	Suppose $V$ is a locally convex Hausdorff t.v.s. with a continuous semi-norm $p$. Let $U$ be an approximatively compact (refer to p) and convex subset of $V$. Consider a map $Pr:V\rightarrow 2^{U}$  defined as,
	$$Pr(v)= \{\bar u\in U|\, p(v-\bar u)=\displaystyle \inf_{u\in U} p(v-u)=d_p(v,U)\}.$$
	Then the map $Pr$ known as best approximation map meets the upper semi-continuity property with $Pr(v)$ being non-empty, compact and convex set for any $v \in V$.
\end{lemma}

\section{Main results}\label{section4}

\subsection{Existence of Best Approximate Solutions for GNEP}\label{section4.1}

This section consists of the main result which ensures the existence of best approximate solution for $GNEP(F_i,u_i)_{i\in I}$ under some sufficient conditions on the objective functions $u_i$ and the feasible strategy maps $F_i$. Moreover, it is worth mentioning that in our proof generalized maximum theorem due to Tian and Zhou \cite{tian} and an important generalization of Kakutani fixed point theorem due to M. Lassonde \cite{lassonde} will play a significant role.

\begin{theorem}\label{theorem3}
	Assume that $I$ is an index set which is (possibly) uncountable. Suppose for $i\in I$, $Y_i$ is locally convex Hausdorff t.v.s. consisting a continuous semi-norm $p_i$ and $X_i\subseteq Y_i$ is non-empty, convex and approximatively compact (refer to $p_i$). Suppose for every $i\in I$,
	\begin{enumerate}
		\item[(a)] $u_i:Y_{-i}\times Y_i \rightarrow \mathbb{R}$ is upper semi-continuous real valued function which is quasi-concave in $y_{i}$;
		\item[(b)] $F_i:X_{-i}\rightarrow 2^{Y_{i}}$ is non-empty, closed and convex valued upper semi-continuous mapping where $\overline{F_i(X_{-i})}$ is compact;
		\item[(c)] $u_i$ is FPT l.s.c. in $x_{-i}$ refer to $F_{i}$.
	\end{enumerate}
	Then the generalized Nash game $GNEP(F_i,u_i)_{i\in I}$ admits atleast one best approximate solution. 
\end{theorem}

\begin{proof}
	For any $i\in I$, we construct  $M_{i}:X_{-i}\rightarrow 2^{Y_{i}}$  as follows,
	\begin{center}
		$M_i(x_{-i})=\{y_i\in F_i(x_{-i})|\, u_i(x_{-i},y_{i})\geq u_i(x_{-i},z_{i})~\textrm{for~any}~ z_{i}\in F_i(x_{-i})\}$.
	\end{center}
	
	We can derive that $M_i(x_{-i})$ is convex. Taking any two elements $w_i,z_i\in M_i(x_{-i})\subset F_i(x_{-i})$ and $t \in  [0,1]$ arbitrary, it occurs $tw_i+(1-t)z_i \in F_i(x_{-i})$ as $F_i$ is convex-valued. Since $u_i$ is quasi-concave w.r.t. $y_i$, it appears
	\begin{align*}
	u_i(x_{-i},tw_i+(1-t)z_i) &\geq \min\{u_i(x_{-i},w_i),u_i(x_{-i},z_i)\}\\
	&\geq u_i(x_{-i},s_i),
	\end{align*}
	for any $s_i \in F_i(x_{-i})$. Hence, $tw_i+(1-t)z_i \in M_i(x_{-i})$ and $M_i(x_{-i})$ is convex.
	
	Clearly, $F_i(x_{-i})$ is compact for every $x_{-i}\in X_{-i}$. Thus, applying Theorem \ref{theorem1} for these chosen $F_i,u_i$, we see that the correspondence $M_i$ fulfills the upper semi-continuity property and the set $M_i(x_{-i})$ becomes non-empty and compact.
	
	Therefore, it appears from \cite[Theorem 17.28]{aliprantis} that the map $M:X\rightarrow 2^Y$ formed as $M(x)=\prod_{i\in I} M_i(x_{-i})$ meets the upper semi-continuity property and the set $M(x)$ becomes non-empty, compact and convex.
	
	Construct the best approximation map $Pr_i:Y_i\rightarrow 2^{X_i}$ for every $i\in I$ as follows,
	$$Pr_i(y_i)= \{\bar x_i\in X_i|\, p_i(y_i-\bar x_i)=  \displaystyle \inf_{x_i\in X_i} p_i(y_i-x_i)\}.$$
	Then according to Lemma \ref{lemma1}, it occurs $Pr_i$ meets the upper semi-continuity property and the set $Pr_i(y_i)$ becomes non-empty, convex and compact. Consequently, by virtue of \cite[Theorem 17.28]{aliprantis}, the map $Pr:Y\rightarrow 2^X$ formed as $Pr(y)= \prod_{i\in I} Pr_i(y_i)$ becomes upper semi-continuous and the set $Pr(y)$ is non-empty, convex and compact.
	
	Let us define the map $\Gamma=Pr\circ M:X\rightarrow 2^X$. We aim to show that the set $\overline{\Gamma(X)}$ eventually becomes compact. Due to the fact that $\overline{F_i(X_{-i})}$ is compact set for every $i\in I$ and $\overline{M(X)}\subset \prod_{i\in I} \overline{F_i(X_{-i})}$, we conclude that $\overline{M(X)}$ is compact.
	Therefore $Pr(\overline{M(X)})$ is compact as $Pr$ is compact valued meeting the upper semi-continuity property.
	
	
	Therefore we see that the map $\Gamma=Pr\circ M$ meets all the required conditions of Theorem \ref{theorem2}. Therefore an element $\tilde{x}\in X$ occurs with $\tilde{x}\in \Gamma(\tilde{x})= Pr\circ M(\tilde{x})$, according to Theorem \ref{theorem2}. We conclude that this $\tilde{x}\in X$ is eventually a best approximate solution for the given generalized Nash game. Since $\tilde{x} \in Pr \circ M(\tilde x)$, there exists $\tilde{y} \in M(\tilde{x})$ with $\tilde{x} \in Pr(\tilde{y}$).
	Thus  $\tilde{x}=(\tilde x_i)_{i \in I}$ lies in $X$ and $\tilde{y}=(\tilde y_i)_{i \in I}$ lies in $Y$ satisfying the following for any $ i$:
	$$p_i(\tilde y_i-\tilde x_i)= \displaystyle \inf_{x_i\in X_i} p_i(\tilde y_i-x_i), ~\textrm{and}$$
	$$u_i(\tilde{x}_{-i},\tilde y_i) \geq u_i(\tilde{x}_{-i},y_i),~ \textrm{for~any}~y_i \in F_i(\tilde{x}_{-i}).$$
\end{proof}
\begin{remark}\leavevmode
	\begin{itemize}
		\item[(i)] We obtain Theorem \ref{theorem6} due to Tian and Zhou \cite{tian} as a corollary of Theorem \ref{theorem3}, on considering the product of feasible strategy maps, $F=\prod_{i\in I} F_i$ to be self-map, that is, $F_i(X_{-i})\subseteq X_i$ for each $i\in I$.
		\item[(ii)]	In Theorem \ref{theorem3} for any $i\in I$, consider $u_i$ to be continuous and assumptions over $F_i$ are intact. If we restrict each $X_i \neq \emptyset$ to be convex, compact subset of $\mathbb{R}^{n_i}$ and $F_i(X_{-i})\subseteq X_i$ for each $i\in I$ (that is $F$ is a self map). Then we obtain the result equivalent to the existence result \cite[Theorem 4.3.1]{ichiishi}) due to Ichiishi.
		\item [(iii)] We proved the existence of best approximate solution for $GNEP(F_i,u_i)_{i\in I}$ considering quasi-concavity and weak continuity (feasible path transfer lower semi-continuity) of objective functions which are obviously preferable over the convex differentiable objective function as considered in \cite{aussel2016} to prove the existence of projected solution.
	\end{itemize}
\end{remark}
Following existence result for $GNEP(F_i,u_i)_{i\in I}$ defined over infinite dimensional uniformly convex Banach spaces having unbounded strategy sets can be obtained by combining Theorem \ref{theorem3} and Example \ref{ex2},
\begin{corollary}\label{coro2}
	Suppose for any $i\in I$, $Y_i$ is uniformly convex Banach space and $ X_i\subseteq Y_i $ is non-empty, closed and convex. If for each $i\in I$,
	\begin{enumerate}
		\item[a)]$u_i:Y_{-i}\times Y_i \rightarrow \mathbb{R}$ is upper semi-continuous real valued function which is quasi-concave in the variable $y_{i}$;
		\item[b)]$F_i:X_{-i}\rightarrow 2^{Y_{i}}$ is non-empty, closed and convex valued upper semi-continuous map where the set $\overline{F_i(X_{-i})}$ is compact;
		\item[c)]$u_i$ is FPT l.s.c. in $x_{-i}$ refer to $F_{i}$.
	\end{enumerate}
	Then $GNEP(F_i,u_i)_{i\in I}$ admits atleast one best approximate solution.
\end{corollary}


\subsection{Existence of Best Approximate Solutions for Quasi-Optimization Problems}\label{section4.2}
This section consists of the existence result concerned with best approximate solutions for quasi-optimization problems. The quasi-optimization problem is generally accompanied by a constraint map which we consider to be a non-self map in the present work.

Assume that $U$ is a non-empty subset of locally convex Hausdorff t.v.s. $V$. Suppose $f:U\rightarrow \mathbb{R}$ is a real valued function and $K:U\rightarrow 2^U$ is a constraint map. For the considered $f$ and $K$, the  quasi-optimization problem (see \cite{Qopt}) is to obtain $\bar u\in U$ satisfying the following,
$$ \bar u\in K(\bar u) \text{ and } f(\bar u)=\displaystyle\max_{w\in K(\bar u)}f(w).$$

Some of the existence results related to solution of quasi-optimization problems can be found under \cite{ausselopt,cotrina}. The recent advances includes the occurrence of projected solution for quasi-optimization problem defined over finite dimensional space (refer \cite{aussel2016}). We have extended this concept to the best approximate solution for quasi-optimization problem defined over locally convex Hausdorff t.v.s.

\begin{definition}
	Let $V$ be locally convex Hausdorff t.v.s. with continuous semi-norm $p$ and $U$ be non-empty subset of $V$. For real valued objective function $f:V\rightarrow\mathbb{R}$ and non-self constraint map $K:U\rightarrow 2^V$, an element $\bar u\in U$ is said to be best approximate solution for quasi-optimization problem $QuOp(K,f)$ iff there exists $\overline{v}\in V$ satisfying,
	
	$$p(\overline v-\overline u)=\displaystyle \inf_{u\in U} p(\overline v- u) ~\textrm{and}~\ f(\overline v)=\displaystyle\max_{w\in K(\overline u)} f(w).$$
\end{definition}

Clearly, every (classical) solution $\overline u\in U$ is a best approximate solution for the considered quasi-optimization problem $QuOp(K,f)$, as we may assume $\overline v=\overline u$ in above definition. However, the existence of (classical) solution is only possible if $K(U)\cap U\neq \emptyset$. On the other hand, every best approximate solution is a (classical) solution for the given quasi-optimization problem if $K(U)\subseteq U$.

\begin{remark}
	Note that if the quasi-optimization problem $QuOp(K,f)$ is defined over Euclidean space then the concept of best approximate solution coincides with the concept of projected solution defined in \cite[section 4.1]{aussel2016}.
\end{remark}
Below stated rephrased version of \textit{Berge's Maximum Theorem} \cite[Theorem 2.3.1]{ichiishi} will  play an indispensable role in order to establish the existence of best approximate solution for $QuOp(K,f)$.
\begin{theorem}\label{theorem4} For given topological spaces U and V suppose $f:V\rightarrow \mathbb{R}$ is a continuous function; $K:U\rightarrow 2^V$ is a continuous map with non-empty compact values. Then the maximizing map $M:U\rightarrow 2^V$ defined as,
	$$M(u)=\{v\in K(u): f(v)\geq f(w)~\text{for any}~w\in K(u)\}$$ meets the upper semi-continuity property with non-empty compact values.
\end{theorem}

An existence result which is stated below ensures the presence of the best approximate solution for a quasi-optimization problem $QuOp(K,f)$ under the given sufficient conditions.
\begin{theorem}\label{theorem5}
	Assume $V$ is a locally convex Hausdorff t.v.s. with continuous semi-norm $p$ and $U\subseteq V$ is non-empty, convex and approximatively compact (refer to p). Suppose,
	\begin{itemize}
		\item[a)] $f:V\rightarrow \mathbb{R}$ is a continuous, quasi-concave function;
		\item[b)] $K:U\rightarrow2^V$ is non-empty, closed and convex valued continuous mapping where $\overline{K(U)}$ compact.
	\end{itemize}
	Then the quasi-optimization problem $QuOp(K,f)$ admits atleast one best approximate solution.
\end{theorem}
\begin{proof}
	Let us construct the maximizing map, $M:U\rightarrow 2^V$ as follows,
	$$M(u)=\{v\in K(u)|\,f(v)\geq f(w) \ \text{for any } w\in K(u)\}.$$
	Clearly $K(u)$ is compact for every $u\in U$. Thus applying Theorem \ref{theorem4} for the chosen mappings $K,f$, we can observe that the corresponding mapping $M$ fulfills the upper semi-continuity property and the set $M(u)$ becomes non-empty and compact.
	
	Following the same argument as presented in proof of Theorem \ref{theorem3}, the set $M(u)$ is convex for any $u\in U$. Indeed if $v,w\in M(u)$ and $s\in [0,1]$ are chosen arbitrarily then by convexity of $K(u)$ and quasi-concavity of $f$ it is evident that, $sx+(1-s)y\in K(u)$ and
	$$f(sv+(1-s)w)\geq \min \{f(v),f(w)\}\geq f(z) $$ for any $z\in K(u)$.
	
	Consider the best approximation map $Pr:V\rightarrow 2^U$ defined as,
	$$Pr(v)=\{\bar u\in U:p(v-\bar u)=\displaystyle \inf_{u\in U}p(v-u)\}.$$
	Then by the virtue of Lemma \ref{lemma1}, the map $Pr$ meets the upper semi-continuity property and the set $Pr(v)$ becomes non-empty, compact and convex.
	
	Now we claim that for the defined map $\Gamma= Pr\circ M:U\rightarrow 2^U$, the set $\overline{\Gamma (U)}$ is compact. It is noticeable that $\overline{M(U)}\subset \overline{K(U)}$ is compact. Thus the set $\overline{\Gamma(U)}\subseteq Pr(\overline{M(U)})$ is compact as $Pr$ is compact valued map meeting the upper semi-continuity property.
	
	It can be observed that the map $\Gamma=Pr\circ M$ fulfills the required conditions of Theorem \ref{theorem2} which leads us to the existence of $\overline u\in U$ such that, $\overline u\in Pr\circ M(\overline u)$. We assert that $\overline u$ is best approximate solution for the proposed optimization problem $QuOp(K,f)$. Since $\overline u\in Pr\circ M(\overline u)$, there exists $\overline v\in M(\overline u)$ with $\overline u\in Pr(\overline v)$. Thus $\overline u\in U$ and $\overline v\in V$ fulfills the following condition,
	$$p(\overline v-\overline u)=\displaystyle \inf_{u\in U}p(\overline v-u)~\text{and}~
	f(\overline v)=\displaystyle\max_{w\in K(\overline u)} f(w).$$
\end{proof}
Following existence result for quasi-optimization problem defined over uniformly convex Banach spaces can be obtained on combining Theorem \ref{theorem5} and Example \ref{ex2},
\begin{corollary}\label{coro4}
	Assume that $U$ is a non-empty, closed and convex subset of uniformly convex Banach space $V$. Suppose,
	\begin{enumerate}
		\item [a)] $f:V\rightarrow \mathbb{R}$ is continuous and quasi-concave function;
		\item [b)] $K:U\rightarrow 2^V$ is non-empty, closed and convex valued upper semi-continuous mapping where $\overline{K(U)}$ is compact.
	\end{enumerate} Then there exists a best approximate solution for $QuOp(K,f)$.
\end{corollary}

\begin{remark}
	We know that the best approximate solution for $QuOp(K,f)$ is a (classical) solution if the constraint map $K$ is a self map. Consequently, for some convex and compact set $U\subseteq \mathbb{R}^n$, if $K$ is considered to be a self constraint map, then \cite[Proposition 4.5]{ausselopt} follows from Theorem \ref{theorem5}.
\end{remark}

Moreover, the following result related to existence of classical solution for $QuOp(K,f)$ is direct consequence of Theorem \ref{theorem5}.
\begin{corollary}\label{coro3} \cite[Corollary 3.2]{cotrina}
	Assume that $U$ is a convex compact subset of some locally convex Hausdorff t.v.s. Suppose $f:U\rightarrow\mathbb{R}$ is a continuous, quasi-convex function; $K:U\rightarrow 2^U$ is lower semi-continuous map with convex values and closed graph. Then there exists $\overline{u}\in U$ satisfying,
	$$\overline u \in K(\overline u)~\textrm{and}~f(\overline u)=\displaystyle\min_{w\in U} f(w).$$
\end{corollary}

\section*{Acknowledgment}
The second author is grateful to University Grants Commission (UGC), New Delhi, India for providing the financial aid to carry out this research work under the enrollment number $\big(\rm{1313/(CSIRNETJUNE2019)}\big)$.


\begin{thebibliography}{99}
	\bibitem{aliprantis} C.D. Aliprantis and K.C. Border, \textit{Infinite Dimensional Analysis: A Hitchhiker's Guide}, Springer, Berlin, 2006.
	%
	\bibitem{debreu} K.J. Arrow and G. Debreu, Existence of an equilibrium for a competitive economy, \textit{Econometrica} 22 (1954) 265--290.
	%
	\bibitem{ausselopt} D. Aussel and J. Cotrina, Quasimonotone quasivariational inequalities: existence results and applications, \textit{J. Optim. Theory Appl.} 158 (2013) 637--652.
	%
	\bibitem{ausselTDGNEP} D. Aussel, R. Gupta and A. Mehra, Evolutionary variational inequality formulation of the generalized Nash equilibrium problem, \textit{J. Optim. Theory Appl.} 169 (2016) 74--90.
	%
	\bibitem{aussel2016} D. Aussel, A. Sultana and V. Vetrivel, On the existence of projected solutions of quasi-variational inequalities and generalized Nash equilibrium problems, \textit{J. Optim. Theory Appl.} 170 (2016) 818--837.
	%
	\bibitem{berge} C. Berge, \textit{Topological Spaces: Including a treatment of multi-valued functions, vector spaces and convexity}, Oliver and Boyd, Edinburgh and London, 1963.
	%
	\bibitem{vetrivel} P. Bhattacharyya and V. Vetrivel, An existence theorem on generalized quasi-variational inequality problem, \textit{J. Math. Anal. Appl.} 188 (1994) 610--615.
	%
	\bibitem{cotrinaGNEP} O. Bueno and J. Cotrina, Existence of projected solutions for generalized Nash equilibrium problems, \textit{J. Optim. Theory Appl.} 191 (2021) 344--362.
	%
	\bibitem{cotrina} J. Cotrina and J. Z{\'u}{\~n}iga, A note on quasi-equilibrium problems, \textit{Oper. Res. Lett.} 46 (2018) 138--140.
	%
	\bibitem{debreu1952} G. Debreu, A social equilibrium existence theorem, \textit{Proc. Nat. Acad. Sci.} 38 (1952) 886--893.
	%
	\bibitem{faccheini} F. Facchinei and C. Kanzow, Generalized Nash equilibrium problems, \textit{Ann. Oper. Res.} 175 (2010) 177--211.
	%
	\bibitem{fischer} A. Fischer, M. Herrich and K. Sch{\"o}nefeld, Generalized Nash equilibrium problems-recent advances and challenges, \textit{Pesquisa Operacional} 34 (2014) 521-558.
	%
	\bibitem{Qopt} F. Giannessi, G. Mastroeni and L. Pellegrini, \textit{Vector Variational Inequalities and Vector Equilibria}, Springer, Boston, 2000.
	%
	\bibitem{ichiishi} T. Ichiishi, \textit{Game Theory for Economic Analysis}, Academic Press, New York, 1983.
	%
	\bibitem{lassonde} M. Lassonde, Fixed points for Kakutani factorizable multifunctions, \textit{J. Math. Anal. Appl.} 152 (1990) 46--60.
	%
	\bibitem{nash1} J. Nash, Equilibrium points in n-person games, \textit{Proc. Nat. Acad. Sci.} 36 (1950) 48--49.
	%
	\bibitem{ramos} M. Ramos, M. Bo\c{i}x, D. Aussel, L. Montastruc and S. Domenech, Water integration in eco-industrial parks using a multi-leader-follower approach, \textit{Comp. Chem. Engg.} 87 (2016) 190--207.
	%
	\bibitem{singh} V.M. Sehgal and S.P. Singh, A generalization to multifunctions of Fan's best approximation theorem, \textit{Proc. Amer. Math. Soc.} 102 (1998) 534--537.
	%
	\bibitem{tian} G. Tian and J. Zhou, The maximum theorem and the existence of Nash equilibrium of (generalized) games without lower semi-continuities, \textit{J. Math. Anal. Appl.} 166 (1992) 351--364.
	%
	\bibitem{yannels} N. C. Yannelis and N. D. Prabhakar, Existence of maximal elements and equilibria in linear topological spaces, \textit{J. Math. Econ.} 12 (1983) 223--245.

\end{thebibliography}
\end{document}